\documentclass{amsart}
\usepackage{amsmath, amssymb, mathrsfs, verbatim, multirow}
\usepackage[all]{xy}

\newtheorem{Teo}{Theorem}[section]
\newtheorem{Prop}[Teo]{Proposition}

\newtheorem{Cor}[Teo]{Corollary}

\theoremstyle{definition}
\newtheorem{Def}[Teo]{Definition}

\newtheorem{Obs}[Teo]{Remark}

\newcommand{\N}{\mathbb{N}}

\newcommand{\lra}{\longrightarrow}

\newcommand{\VR}{\mathcal{O}}

\newcommand{\MI}{\mathfrak{m}}

\newcommand{\SU}{\mbox{\rm supp}}

\begin{document}
\title{Key polynomials and minimal pairs}
\author{Josnei Novacoski}
\thanks{During part of the realization of this project the author was supported by a grant from Funda\c c\~ao de Amparo \`a Pesquisa do Estado de S\~ao Paulo (process number 2015/23409-7).}

\begin{abstract}
In this paper we establish the relation between key polynomials (as defined in \cite{SopivNova}) and minimal pairs of definition of a valuation. We also discuss truncations of valuations on a polynomial ring $K[x]$. We prove that a valuation $\nu$ is equal to its truncation on some polynomial if and only if $\nu$ is valuation-transcendental. Another important result of this paper is that if $\mu$ is any extension of $\nu$ to $\overline K[x]$ and $\Lambda$ is a complete sequence of key polynomials for $\nu$, without last element, then for each $Q\in \Lambda$ there exists a suitable root $a_Q\in \overline K$ of $Q$ such that $\{a_Q\}_{Q\in \Lambda}$ is a pseudo-convergent sequence defining $\mu$.
\end{abstract}

\maketitle
\section{Introduction}
Given a valuation $\nu$ on a field $K$, it is important to understand the possible extensions of $\nu$ to  $K[x]$, the polynomial ring in one indeterminate over $K$. This problem has been extensively studied and many objects have been introduced to describe such extensions. For instance, MacLane introduced in \cite{Mac} the concept of \textit{key polynomials} and proved that in the discrete case, every extension of $\nu$ to $K[x]$ is determined by a (possibly infinite) sequence of key polynomials. In \cite{Vaq}, Vaqui\'e generalized this result, proving that every valuation is defined by a sequence of key polynomials, as long as we admit a special type of key polynomials without immediate predecessor, which are called \textit{limit key polynomials}. The definition of key polynomials by MacLane and Vaqui\'e is recursive, in the sense that they serve to augment a given valuation (or a sequence of valuations in the case of limit key polynomials).

An alternative definition of key polynomials was introduced in \cite{SopivNova}. This definition differs from the MacLane--Vaqui\'e's in a very essential way: it gives a way to truncate a valuation instead of augmenting it. This property gives many advantages when dealing with valuations. For instance, it allows us to determine whether a polynomial is a key polynomial, without having to define recursively a preceding valuation (or sequence of valuations). It also gives a more natural way to handle applications. For instance, in the local uniformization problem, it provides a way of proving results for a given truncation of the valuation (for more details see the appendix of this paper). A comparison between the MacLane--Vaqui\'e key polynomials and the one presented in \cite{SopivNova} can be found in \cite{Spivmahandjul}.

Two other objects used to study extensions of valuations are pseudo-convergent sequences, introduced by Kaplansky in \cite{Kapl} and minimal pairs introduced by Alexandru, Popescu and Zaharescu in \cite{Kand1} and \cite{Kand2}. These objects have importance on their own. For instance, pseudo-convergent sequences are used (implicitly) by Knaf and  Kuhlmann in \cite{KK} to prove that every place admits local uniformization in a finite extension of the function field. On the other hand, minimal pairs were used in \cite{Kand1} to prove Nagata's conjecture on the structure of the residue field extension of a \textit{residually transcendental extension}.

The first goal of this paper is to establish the relation between minimal pairs and key polynomials. Fix a valuation $\nu$ on $K[x]$, the polynomial ring in one indeterminate over the field $K$. We fix an algebraic closure $\overline K$ of $K$ and an extension $\mu$ of $\nu$ to $\overline K[x]$. For a monic polynomial $f\in K[x]$, we define
\[
\delta(f):=\max\{\mu(x-a)\mid a\mbox{ is a root of }f\}.
\]
A \textbf{minimal pair for $\nu$} is a pair $(a,\delta)\in\overline K \times \mu(\overline K[x])$ such that for every $b\in\overline K$, if $\mu(b-a)\geq \delta$, then $[K(b):K]\geq [K(a):K]$. If in addition, $\mu(x-a)=\delta\geq \mu(x-b)$ for every $b\in \overline K$, then $(a,\delta)$ is called a \textbf{minimal pair of definition for $\nu$}.

For a positive integer $r$ and $f\in K[x]$ let $\partial_r f$ be the \textbf{$r$-th formal derivative} of  $f$, i.e., $\partial_r f$ are the uniquely determined polynomials for which the Taylor expansion
\[
f(x)-f(a)=\sum_{i=1}^{\deg(f)}\partial_if(a)(x-a)^i,
\]
is satisfied for every $a\in K$. For a polynomial $f\in K[x]$ let
\[
\epsilon(f)=\max_{r\in \N}\left\{\frac{\nu(f)-\nu(\partial_rf)}{r}\right\}.
\]
A monic polynomial $Q\in K[x]$ is said to be a \textbf{key polynomial} (of level $\epsilon (Q)$) if for every $f\in K[x]$ if $\epsilon(f)\geq \epsilon(Q)$, then $\deg(f)\geq\deg(Q)$.

Given two polynomials $p,q\in K[x]$, there exist uniquely determined $p_0,\ldots, p_n\in K[x]$ with
\begin{equation}\label{standdecompofp}
p=p_0+p_1q+\ldots+p_nq^n
\end{equation}
such that for each $i$, $p_i=0$ or $\deg(p_i)<\deg(q)$. Equation (\ref{standdecompofp}) will be called the \textbf{$q$-expansion of $p$}. Given a valuation $\nu$ on $K[x]$ and a polynomial $q\in K[x]$ we can define the mapping
\[
\nu_q(p):=\min_{0\leq i\leq n}\{\nu(p_iq^i)\}.
\]
This mapping is called the \textbf{$q$-truncation of $\nu$}. It is natural to ask when $\nu_q$ is a valuation. In \cite{SopivNova}, we prove that if $Q$ is a key polynomial for $\nu$, then $\nu_Q$ is a valuation on $K[x]$.

The main result of this paper is the following:
\begin{Teo}\label{main1}
Let $Q\in K[x]$ be a monic irreducible polynomial and choose a root $a$ of $Q$ such that $\mu(x-a)=\delta(Q)$. Then $Q$ is a key polynomial for $\nu$ if and only if $(a,\delta(Q))$ is a minimal pair for $\nu$. Moreover, $(a,\delta(Q))$ is a minimal pair of definition for $\nu$ if and only if $\nu=\nu_Q$.
\end{Teo}

We also focus on relating key polynomials and pseudo-convergent sequences. A set of key polynomials $\Lambda$ with $\epsilon(Q)\neq \epsilon (Q')$ for $Q,Q'\in \Lambda$ with $Q\neq Q'$ is called a \textbf{complete sequence of key polynomials} if it is well-ordered (with the order $Q<Q'$ if and only if $\epsilon(Q)<\epsilon(Q')$) and for every $f\in K[x]$ there exists $Q\in \Lambda$ such that $\nu_Q(f)=\nu(f)$. In Theorem 1.1 of \cite{SopivNova} we prove that every valuation on $K[x]$ admits a complete sequence of key polynomials.

A \textbf{pseudo-convergent sequence} for $\nu$ is a well-ordered subset $\{a_{\rho}\}_{\rho<\lambda}$ of $K$, without last element, such that
\begin{equation}\label{eqdefipcs}
\nu(a_\sigma-a_\rho)<\nu(a_\tau-a_\sigma)\mbox{ for all }\rho<\sigma<\tau<\lambda
\end{equation}
(observe that this definition can be extended for any valued field $(K,\nu)$). An element $a\in K[x]$ is said to be a \textbf{limit} of the pseudo-convergent sequence $\{a_\rho\}_{\rho<\lambda}$ if $\nu(a-a_\rho)=\nu(a_{\rho+1}-a_\rho)$ for every $\rho<\lambda$. One can prove that for every polynomial $f(x)\in K[x]$, there exists $\rho_f<\lambda$ such that either
\begin{equation}\label{condforpscstotra}
\nu(f(a_\sigma))=\nu(f(a_{\rho_f}))\mbox{ for every }\rho_f\leq \sigma<\lambda,
\end{equation}
or
\begin{equation}\label{condforpscstoalg}
\nu(f(a_\sigma))>\nu(f(a_{\rho}))\mbox{ for every }\rho_f\leq \rho< \sigma<\lambda.
\end{equation}
If case (\ref{condforpscstotra}) happens, we say that the value of $f$ is fixed by $\{a_\rho\}_{\rho<\lambda}$ (or that $\{a_\rho\}_{\rho<\lambda}$ fixes the value of $f$). A pseudo-convergent sequence $\{a_\rho\}_{\rho<\lambda}$ is said to be of \textbf{transcendental type} if for every polynomial $f(x)\in K[x]$ the condition (\ref{condforpscstotra}) holds.
Otherwise, $\{a_\rho\}_{\rho<\lambda}$ is said to be of \textbf{algebraic type}, i.e., if there exists at least one polynomial for which condition (\ref{condforpscstoalg}) holds. In Theorem 1.2 of \cite{SopivNova}, we show how to obtain a sequence of key polynomials from a pseudo-convergent sequence and how to understand the different cases (algebraic and transcendental). In this paper we present the converse of that, i.e., we prove the following:
\begin{Prop}\label{propconvofnovaspiv}
Let $\nu$ be a valuation of $K[x]$ and let $\mu$ be an extension of $\nu$ to $\overline K[x]$. Let $\Lambda$ be a complete sequence of key polynomials for $\nu$, without last element. For each $Q\in\Lambda$, let $a_Q\in \overline K$ be a root of $Q$ such that $\mu(x-a_Q)=\delta(Q)$. Then $\{a_Q\}_{Q\in \Lambda}$ is a pseudo-convergent sequence of transcendental type, without a limit in $\overline K$, such that $x$ is a limit for it.
\end{Prop}
In Theorem 2 of \cite{Kapl}, Kaplansky proves that if $\nu$ is a valuation of $K$, and $\{a_\rho\}_{\rho<\lambda}$ is a pseudo-convergent sequence of transcendental type, without a limit on $K$, then $\{a_\rho\}_{\rho<\lambda}$ uniquely defines a valuation (also called $\nu$) on $K[x]$ (by setting $\nu(p(x))$ to be the fixed value of $p(x)\in K[x]$ by $\{a_\rho\}_{\rho<\lambda}$). Hence, Proposition \ref{propconvofnovaspiv} tells us that $\mu$ is determined by $\{a_Q\}_{Q\in \Lambda}$ in $\overline K[x]$.

The main tool to prove Theorem \ref{main1} and Proposition \ref{propconvofnovaspiv} is Proposition \ref{Profmagica}. This proposition tells us that for every monic polynomial $f\in K[x]$ we have $\epsilon(f)=\delta(f)$. This result is important on its own because it provides a more intuitive characterization of $\epsilon(f)$.

In this paper, we also study when a valuation is defined by finitely many (or equivalently, by one) key polynomials. Our first result is that the property $\nu=\nu_q$ does not characterize key polynomials, i.e., that there exists a valuation $\nu$ on $K[x]$ and a polynomial $q\in K[x]$ such that $\nu=\nu_q$ and $q$ is not a key polynomial for $\nu$ (see Corollary \ref{Corthatgivexamplenotconj}).

A valuation $\nu$ on $K[x]$ is called \textbf{valuation-algebraic} if $\nu(K(x))/\nu K$ is a torsion group and $K(x)\nu\mid K\nu$ is an algebraic extension. Otherwise, it is called \textbf{valuation-transcendental} (for more details, see Section \ref{Preliminaires}). Another important result of this paper is the following:
\begin{Teo}\label{Thmsobretruncvalutrans}
A valuation $\nu$ on $K[x]$ is valuation-transcendental if and only if there exists a polynomial $q\in K[x]$ such that $\nu=\nu_q$.
\end{Teo}
The theorem above can be seen as the version of Theorem 3.11 of \cite{Kuhbadplaces} for key polynomials and truncations. Observe that if $\Lambda$ is a complete sequence of key polynomials for $\nu$ with last element $Q$, then $\nu=\nu_Q$. Hence, we conclude from Theorem \ref{Thmsobretruncvalutrans}, that if $\Lambda$ has a last element, then $\mu$ is valuation-transcendental.

This paper is divided as follows. In the next section we present the definitions and basic results which will be used in the sequel. In Section \ref{proofofthemman1}, we present the proofs of Proposition \ref{propconvofnovaspiv} and Theorem \ref{main1} and in Section \ref{Pproofthemouro} the proof of Theorem \ref{Thmsobretruncvalutrans}. We also present an appendix where we discuss the  importance of key polynomials for the local uniformization problem.
\par\medskip
\textbf{Acknowledgements.} I would like to thank Franz-Viktor Kuhlmann for a careful reading, for providing useful suggestions and for pointing out a few mistakes in an earlier version of this paper.

\section{Preliminaries}\label{Preliminaires}

Take a commutative noetherian ring $R$ ($R$ may have zero divisors and even non-zero nilpotent elements) and an abelian group $\Gamma$. Take $\infty$ to be an element not in $\Gamma$ and set $\Gamma_\infty$ to be $\Gamma\cup\{\infty\}$ with extensions of addition and order as usual.

\begin{Def}
A \textbf{valuation on $R$} is a map $\nu:R\lra \Gamma_\infty$ such that the following holds:
\begin{description}
\item[(V1)] $\nu(ab)=\nu(a)+\nu(b)$ for every $a,b\in R$,
\item[(V2)] $\nu(a+b)\geq \min\{\nu(a),\nu(b)\}$ for every $a,b\in R$,
\item[(V3)] $\nu(1)=0$ and $\nu(0)=\infty$,
\item[(V4)] $\SU(\nu):=\{a\in R\mid \nu(a)=\infty\}$ is a minimal prime ideal of $R$.
\end{description}
\end{Def}

Observe that this definition implies that if $\nu(a)\neq \nu(b)$, then
\[
\nu(a+b)=\min\{\nu(a),\nu(b)\}.
\]
Given a valuation $\nu$ on $R$, we define the \textbf{value group of $\nu$} and denote by $\nu R$ to be the subgroup of $\Gamma$ generated by $\{\nu(a)\mid a\in R\setminus \SU(\nu)\}$. If $R$ is a field, then we define the \textbf{valuation ring of $\nu$} by $\VR_\nu:=\{a\in R\mid \nu(a)\geq 0\}$. The ring $\VR_\nu$ is a local ring with unique maximal ideal $\MI_\nu:=\{a\in R\mid \nu(a)>0\}$. In this case, we define the \textbf{residue field of $\nu$}, denoted by denote by $R\nu$ to be the field $\VR_\nu/\MI_\nu$ (for $f\in R$ we denote by $f\nu$ the image of $f$ in $R\nu$).

Throughout this paper we will fix a field $K$, a valuation $\nu$ on $K[x]$, the polynomial ring in one indeterminate over $K$, an algebraic closure $\overline K$ of $K$ and an extension $\mu$ of $\nu$ to $\overline K[x]$.

\begin{Obs}
The valuations $\nu$ as above describe all the valuations extending $\nu_0=\nu|_K$ to simple extensions $K(a)$ of $K$. Indeed, if $\SU(\nu):=\{p\in K[x]\mid\nu(p)=\infty\}$ is the zero ideal, then $\nu$ extends in an obvious way to $K(x)$ where $x$ is a transcendental element. If $\SU(\nu)\neq (0)$, then there exists $p(x)\in K[x]$ monic and irreducible such that $\SU(\nu)= (p)$. Hence, $\nu$ defines a valuation on $K[x]/(p)=K(a)$ for some element $a\in\overline K$ with minimal polynomial $p(x)$.
\end{Obs}

The case when $\SU(\nu)=(p)\neq (0)$ is trivial for our purposes, since $p$ will be a key polynomial and $\nu=\nu_p$. Hence, we assume from now on, that $\SU(\nu)= (0)$. In this case, $\nu$ can be extended to $K(x)$ and we can consider the residue field extension $K(x)\nu|K\nu$. The valuation $\nu$ is called \textbf{value-transcendental} if there exists $f\in K(x)$ such that $\nu(f)$ is torsion-free over $\nu K$. On the other hand, $\nu$ is called \textbf{residue-transcendental} if there exists $f\in \VR_{K(x)}$ ($\VR_{K(x)}:=\{f\in K(x)\mid\nu(f)\geq 0\}$) such that $f\nu$ is transcendental over $K\nu$. It is immediate from the definition that $\nu$ is valuation-transcendental if and only if it is residue-transcendental or value-transcendental.

We proceed now with the discussion of when $\nu_q$ is a valuation. In Example 2.5 of \cite{SopivNova}, we  show that $\nu_q$ does need to be a valuation and that if $Q$ is a key polynomial, then $\nu_Q$ is a valuation (Proposition 2.6). We will show now that polynomials $q$ for which $\nu_q$ is a valuation do not need to be key polynomials.
\begin{Prop}
Let $\nu$ be a valuation of $K[x]$ and $q\in K[x]$ such that $\nu=\nu_q$. Then $\nu=\nu_{q^2}$.
\end{Prop}
\begin{proof}
Take any polynomial $p\in K[x]$ and let
\[
p=p_0+p_1q+\ldots+p_nq^n
\]
be the $q$-expansion of $p$. Assume that $n$ is odd, by setting $p_n=0$ if necessary. Then the $q^2$-expansion of $p$ is
\[
p=(p_0+p_1q)+(p_2+p_3q)q^2+\ldots+(p_{n-1}+p_nq)\left(q^2\right)^{\frac{n-1}{2}}.
\]
Since $\nu=\nu_q$ we have that
\[
\nu\left((p_{2i}+p_{2i+1}q)q^{2i}\right)=\min\{\nu(p_{2i}q^{2i}),\nu(p_{2i+1}q^{2i+1})
\]
and consequently
\begin{eqnarray*}
\nu_{q^{2}}(p)&=&\min_{0\leq i\leq \frac{n-1}{2}}\{\nu\left((p_{2i}+p_{2i+1}q)q^{2i}\right)\}\\
              &=&\min_{0\leq i\leq \frac{n-1}{2}}\{\min\{\nu(p_{2i}q^{2i}),\nu(p_{2i+1}q^{2i+1})\}=\nu_q(p)=\nu(p).
\end{eqnarray*}
\end{proof}

\begin{Cor}\label{Corthatgivexamplenotconj}
There exists a valuation $\nu$ on $K[x]$ and a polynomial $q\in K[x]$ such that $\nu=\nu_q$ but $q$ is not a key polynomial.
\end{Cor}
\begin{proof}
Let $\nu$ be a non-trivial valuation on $K$ and let $\nu$ be the $x$-adic valuation on $K[x]$. This means that $\nu=\nu_x$ and the above proposition guarantees that $\nu=\nu_{x^2}$. Since key polynomials are irreducible (see Proposition 2.4 \textbf{(ii)} of \cite{SopivNova}) $q=x^2$ is not a key polynomial for $\nu$.
\end{proof}

\section{Proofs of Proposition \ref{propconvofnovaspiv} and Theorem \ref{main1}}\label{proofofthemman1}
Before proceeding with the proof of Theorem \ref{main1} we will make a remark about the formal derivatives. Using an independent variable $y$ we can obtain that
\[
f(x+y)=\sum_{i=0}^{\deg(f)}\partial_if(x)y^i\mbox{ where }\partial_0f:=f.
\]
For polynomials $f,g\in K[x]$ we have that
\[
fg(x+y)=f(x+y)g(x+y)=\sum_{i=0}^{\deg(f)+\deg(g)}\left(\sum_{j=0}^i\partial_jf(x)\partial_{i-j}g(x)\right)y^i.
\]
Since the coefficients of the Taylor expansion are uniquely determined, this gives us that for every $r$
\begin{equation}\label{eqleibtaylor}
\partial_r(fg)=\sum_{j=0}^r\partial_jf(x)\partial_{r-j}g(x).
\end{equation}

In order to prove Proposition \ref{propconvofnovaspiv} and Theorem \ref{main1}, we will need the following proposition.
\begin{Prop}\label{Profmagica}
Let $f\in K[x]$ be a monic polynomial. Then $\delta(f)=\epsilon(f)$.
\end{Prop}
\begin{proof}
Let $a_1,\ldots, a_n\in\overline K$ be all the roots (not necessarily distinct) of $f$. Then
\[
f(x)=\prod_{i=1}^n (x-a_i).
\]
For every subset $I$ of $N:=\{1,\ldots,n\}$ we denote by $I^c$ the complement of $I$ in $N$, i.e., $I^c:=N\setminus I$. For every $s$, $1\leq s\leq n$, using interatively Equation (\ref{eqleibtaylor}), we have
\begin{equation}\label{eqderprod}
\partial_s(f)=\sum_{l_1<l_2<\ldots<l_s\leq n}\left(\prod_{i\in \{l_1,\ldots,l_s\}^c}(x-a_i)\right).
\end{equation}
Let
\[
I(f):=\{i\in N\mid \mu(x-a_i)=\delta(f)\}
\]
and $r=|I(f)|$. We claim that
\begin{equation}\label{eqguarmaio}
\nu(\partial_r(f))=\sum_{i\in I(f)^c}\mu(x-a_i).
\end{equation}
Indeed, take any subset $J$ of $N$ with $r$ many elements and assume that $J\neq I(f)$. Write $J^c=J_1\cup J_2$ where $J_1= J^c\cap I(f)$ and $J_2=J^c\cap I(f)^c$. Observe that since $J\neq I(f)$ we have that $J_1\neq \emptyset$. Write $I(f)^c=I_1\cup J_2$. Since $|J^c|=|I(f)^c|=n-r$, we have that $|J_1|=|I_1|$ and moreover
\[
\mu\left(\prod_{i\in J_1}(x-a_i)\right)=|J_1|\delta(f)=|I_1|\delta(f)>\mu\left(\prod_{i\in I_1}(x-a_i)\right).
\]
Hence,
\begin{displaymath}
\begin{array}{rcl}
\displaystyle\mu\left(\prod_{i\in J^c}(x-a_i)\right)&=&\displaystyle\mu\left(\prod_{i\in J_1}(x-a_i)\right)+\mu\left(\prod_{i\in J_2}(x-a_i)\right)\\ \\
&>&\displaystyle\mu\left(\prod_{i\in I_1}(x-a_i)\right)+\mu\left(\prod_{i\in J_2}(x-a_i)\right)\\ \\
&=&\displaystyle\mu\left(\prod_{i\in I(f)^c}(x-a_i)\right).
\end{array}
\end{displaymath}
This and the equality (\ref{eqderprod}) give us the equality (\ref{eqguarmaio}).

We now have that
\[
\nu(f)-\nu(\partial_r(f))=\sum_{i\in N}\mu(x-a_i)-\sum_{i\in I(f)^c}\mu(x-a_i)=\sum_{i\in I(f)}\mu(x-a_i)=r\delta(f).
\]
Hence $\delta(f)\leq \epsilon(f)$. On the other hand, fix $s$, $1\leq s\leq n$ and set $J:=\{l_1,\ldots,l_s\}\subseteq N$ with $l_1<l_2<\ldots<l_s$ for which
\[
\mu\left(\prod_{i\in \{l_1,\ldots,l_s\}^c}(x-a_i)\right)
\]
achieves its minimum. Then
\[
\nu(\partial_sf)\geq \mu\left(\prod_{i\in J^c}(x-a_i)\right)
\]
and consequently
\[
\nu(f)-\nu(\partial_s f)\leq \mu\left(\prod_{i\in J}(x-a_i)\right)\leq s\delta(f).
\]
Therefore, $\epsilon(f)\leq \delta(f)$ which concludes our proof.
\end{proof}
\begin{Obs}
Observe that from the proposition above, we conclude that the number $\delta(f)$ does not depend on the choice of the valuation $\mu$ extending $\nu$.
\end{Obs}

We proceed now to prove Proposition \ref{propconvofnovaspiv}. 

\begin{proof}[Proof of Proposition \ref{propconvofnovaspiv}]
We will start by proving that $\{a_Q\}_{Q\in \Lambda}$ is a pseudo-convergent sequence and that $x$ is a limit of $\{a_Q\}_{Q\in \Lambda}$. For $Q\in \Lambda$, by Proposition \ref{Profmagica} and the assumption on $a_Q$'s, we have
\[
\mu(x-a_{Q})=\delta(Q)=\epsilon(Q).
\]
Let $Q,Q'\in \Lambda$ with $Q'>Q$. Then $\mu(x-a_{Q'})=\epsilon(Q')>\epsilon(Q)=\mu(x-a_Q)$ and consequently
\begin{equation}\label{eqobepsdelta}
\mu(a_{Q'}-a_Q)=\min\{\mu(x-a_{Q'}),\mu(x-a_Q)\}=\mu(x-a_{Q})=\epsilon(Q).
\end{equation}
If $Q_1,Q_2,Q_3\in \Lambda$ with $Q_1<Q_2<Q_3$, then by equation (\ref{eqobepsdelta}) we obtain that
\begin{equation}\label{eqkeypolpcs1}
\mu(a_{Q_3}-a_{Q_2})=\epsilon(Q_2)>\epsilon(Q_1)=\mu(a_{Q_2}-a_{Q_1}).
\end{equation}
Equations (\ref{eqobepsdelta}) and (\ref{eqkeypolpcs1}) imply that $\{a_Q\}_{Q\in\Lambda}$ is a pseudo-convergent sequence and that $x$ is a limit of it.

It remains to show that $\{a_Q\}_{Q\in\Lambda}$ is of transcendental type. Assume that not, so there exists a polynomial $f(x)\in \overline K[x]$ not fixed by $\{a_Q\}_{Q\in\Lambda}$. In particular, there exists an irreducible polynomial not fixed by $\{a_Q\}_{Q\in\Lambda}$ and since $\overline K$ is algebraically closed, this means that there exists $b\in \overline K$ and $Q_0\in K[x]$ such that for every $Q',Q''\in \Lambda$ with $Q''>Q'>Q_0$ we have
\begin{equation}\label{eqfunciona}
\mu(a_{Q''}-b)>\mu(a_{Q'}-b).
\end{equation}
This implies that
\[
\mu(x-b)>\mu(x-a_Q)\mbox{ for every }Q\in \Lambda.
\]
Indeed, if $\nu(x-b)\leq \nu(x-a_Q)$ for some $Q\in \Lambda$, then
\[
\mu(b-a_{Q'})=\min\{\mu(x-b),\mu(x-a_{Q'})\}=\mu(x-b)\mbox{ for every }Q'>Q.
\]
Since $\Lambda$ does not have a last element, there exist $Q',Q''\in \Lambda$ with $Q''>Q'>\max\{Q,Q_0\}$. This implies that $\mu(b-a_{Q''})=\mu(x-b)=\mu(b-a_{Q'})$ which is a contradiction to Equation (\ref{eqfunciona}).

Let $p_b\in K[x]$ be the minimal polynomial of $b$ over $K$. Then
\[
\epsilon(p_b)=\delta(p_b)\geq \mu(x-b)>\mu(x-a_Q)=\epsilon(Q)\mbox{ for every }Q\in \Lambda.
\]
Take $Q'$ to be a monic polynomial of smallest degree in $\{q\in K[x]\mid \epsilon(p_b)\leq \epsilon(q)\}$. Then, by the definition of key polynomial, $Q'$ is a key polynomial with $\epsilon(p_b)\leq\epsilon(Q')$. This implies, by Proposition 2.10 \textbf{(ii)} of \cite{SopivNova} that
\[
\nu_Q(Q')<\nu(Q')\mbox{ for every }Q\in\Lambda
\]
and this is a contradiction to the fact that $\Lambda$ is a complete sequence of key polynomials for $\nu$.
\end{proof}

We will present now the proof of Theorem \ref{main1}.

\begin{proof}[Proof of Theorem \ref{main1}]
Assume first that $Q$ is a key polynomial for $\nu$. Take $b\in \overline K$ such that $\mu(a-b)\geq \delta(Q)$. Since $\mu(x-a)=\delta(Q)$, this implies that $\mu(x-b)\geq \delta(Q)$. Hence $\delta(p_b)\geq \delta(Q)$, where $p_b$ is the minimal polynomial of $b$ over $K$. Proposition \ref{Profmagica} gives us that $\epsilon(p_b)\geq \epsilon(Q)$ and since $Q$ is a key polynomial, this gives us that $\deg(Q)\leq\deg(p_b)$ which implies that $[K(a):K]\leq [K(b):K]$.

For the converse, assume that for every $b\in\overline K$, if $\mu(a-b)\geq \delta(Q)$, then $[K(a):K]\leq [K(b):K]$. We want to prove that for every polynomial $f\in K[x]$, if $\deg(f)<\deg(Q)$, then $\epsilon(f)<\epsilon(Q)$. Since $\deg(f)<\deg(Q)$, by our assumption we obtain that $\mu(a-b)<\delta(Q)$, for every root $b$ of $f$. Hence,
\[
\mu(x-b)=\min\{\mu(x-a),\mu(a-b)\}=\mu(a-b)<\delta(Q).
\]
This implies, using Proposition \ref{Profmagica}, that $\epsilon(f)=\delta(f)<\delta(Q)=\epsilon(Q)$. Therefore, $Q$ is a key polynomial.

For the second part of the proof, assume that $\nu=\nu_Q$. We will show that $\mu(x-a)\geq \mu(x-b)$ for every $b\in \overline K$. Assume, towards a contradiction that there exists $b\in \overline K$ such that $\mu(x-a)<\mu(x-b)$. By Proposition \ref{Profmagica} we have that $\epsilon(Q)<\epsilon(p_b)$, where $p_b$ is the minimal polynomial of $b$ over $K$. Reasoning as in the proof of Proposition \ref{propconvofnovaspiv}, there exists a key polynomial $Q'\in K[x]$ such that $\epsilon(Q')\geq\epsilon(p_b)$. Then we have that $\epsilon(Q)<\epsilon(Q')$ which implies, by Proposition 2.10 \textbf{(ii)} of \cite{SopivNova}, that $\nu_Q(Q')<\nu(Q')$. Therefore, $\nu_Q\neq \nu$, which is a contradiction.

For the converse, assume that $(a,\delta(Q))$ is a minimal pair of definition of $\nu$. We have to show that $\nu_Q=\nu$. If this were not the case, then by Lemma 2.11 of \cite{SopivNova}, there would exist a key polynomial $Q'$ such that $\epsilon(Q)<\epsilon(Q')$. Let $b\in \overline K$ be a root of $Q'$ such that $\mu(x-b)=\delta(Q')$. Then, by Proposition \ref{Profmagica}, we have
\[
\mu(x-b)=\delta(Q')=\epsilon(Q')>\epsilon(Q)=\delta(Q)\geq\mu(x-a)
\]
which is a contradiction.
\end{proof}
\section{Proof of Theorem \ref{Thmsobretruncvalutrans}}\label{Pproofthemouro}

We proceed now with the proof of Theorem \ref{Thmsobretruncvalutrans}.

\begin{proof}[Proof of Theorem \ref{Thmsobretruncvalutrans}]
Assume that $\nu=\nu_q$ for some polynomial $q$. If $\nu q$ is torsion free over $\nu K$, then $\nu$ is value-transcendental and we are done. Hence, we can assume that there exist $a\in K^\times$ and $n\in\N$ such that $\nu\left(\frac{q^n}{a}\right)=0$. This means that $\frac{q^n}{a}\nu\neq 0$. We will prove that $\frac{q^n}{a}\nu$ is transcendental over $K\nu$. Assume that there exist $a_1,\ldots,a_r\in \VR_K$ such that
\[
\sum_{i=0}^r a_i\nu \left(\frac{q^n}{a}\nu\right)^i=0.
\]
Since $\nu=\nu_q$, this implies that for every $j$, $0\leq j\leq r$ we have
\[
\nu\left(a_{j}\left(\frac{q^{n}}{a}\right)^{j}\right)\geq\min_{0\leq i\leq r}\left\{\nu\left(a_i\left(\frac{q^{n}}{a}\right)^i\right)\right\}=\nu\left(\sum_{i=0}^ra_i\left(\frac{q^{n}}{a}\right)^i\right)>0.
\]
Hence,
\[
a_j\nu \left(\frac{q^{n}}{a}\nu\right)^j=\left(a_{j}\left(\frac{q^{n}}{a}\right)^{j}\right)\nu=0,
\]
and since $\frac{q^n}{a}\nu\neq 0$ this implies that $a_j\nu=0$. Therefore, $\frac{q^n}{a}\nu$ is transcendental over $K\nu$ and $\nu$ is residue-transcendental.

Now assume that $\nu$ is value-transcendental. Then there exists $q\in K[x]$ such that $\nu q$ is torsion free over $\nu K$. Choose $q$ with smallest possible degree. This means that for every $p\in K[x]$ with $\deg(p)<\deg(q)$ there exists $n\in \N$ such that $n\nu(p)\in \nu K$. In particular, for every $p_1,p_2\in K[x]$ with $\deg(p_1),\deg(p_2)<\deg(q)$ and for distinct $i_1,i_2\in\N$, we have
\begin{equation}\label{diffevaluecasetrans}
\nu\left(p_1q^{i_1}\right)\neq \nu\left(p_2q^{i_2}\right).
\end{equation}
For every $p\in K[x]$, let $p=p_0+\ldots+p_nq^n$ be the $q$-expansion of $p$. Then Equation (\ref{diffevaluecasetrans}) implies that
\[
\nu(p)=\min_{0\leq i\leq n}\{\nu(p_iq^i)\}=\nu_q(p).
\]

It remains to show that if $\nu$ is residue-transcendental, then there exists $q\in K[x]$ such that $\nu = \nu_q$. Observe that in this case, by the fundamental inequality, $\nu\left (K(x)\right)$ is torsion over $\nu K$. Choose an algebraic closure $\overline K$ of $K$ and an extension $\mu$ of $\nu$ to $\overline K[x]$. Then, for every polynomial $p\in \overline K[x]$, there exists $a\in \overline K$ such that $\mu(ap)=0$.

Assume, towards a contradiction, that for every polynomial $q\in K[x]$ there exists $p\in K[x]$ such that $\nu_q(p)<\nu(p)$. We will show that every element in $K(x)\nu$ is algebraic over $K\nu$. We will start by showing, by induction on the degree, that for every polynomial $q\in K[x]$ and every $a\in\overline K$ with $\mu(aq)= 0$ we have that $\left(aq\right)\mu$ is algebraic over $K\nu$. If $\deg(q)=0$, then our assertion is immediate. Assume now that $\deg(q)>0$ and that for every polynomial $p$ of degree smaller than $\deg(q)$ and every $a\in \overline K$ with $\mu(ap)=0$ we have that $\left(ap\right)\mu$ is algebraic over $K\nu$. By our assumption, there exist $p_0,\ldots,p_r\in K[x]$ such that
\begin{equation}\label{eqabkeypolnotwork}
\nu\left(\sum_{i=0}^rp_iq^i\right)>\min_{0\leq i\leq r}\{\nu\left(p_iq^i\right)\}.
\end{equation}
Take any $a\in\overline K$ such that $\mu(aq)=0$. Adding $\mu(a^r)$ to equation (\ref{eqabkeypolnotwork}) we obtain that
\begin{equation}\label{eqabkeypolnotwork1}
\mu\left(\sum_{i=0}^r(a^{r-i}p_i)(aq)^i\right)>\min_{0\leq i\leq r}\{\mu\left((a^{r-i}p_i)(aq)^i\right)\}=\min_{0\leq i\leq r}\{\mu(a^{r-i}p_i)\}.
\end{equation}
Choose $i_0$ such that
\[
\mu(a^{r-i_0}p_{i_0})=\min_{0\leq i\leq r}\{\mu(a^{r-i}p_i)\}
\]
and $c\in \overline K$ such that $\mu(ca^{r-i_0}p_{i_0})=0$. Then $\mu(ca^{r-i}p_i)\geq 0$ for every $i$. Adding $\mu(c)$ to equation (\ref{eqabkeypolnotwork1}) gives us
\begin{equation}\label{eqabkeypolnotwork2}
\mu\left(\sum_{i=1}^r(ca^{r-i}p_i)(aq)^i\right)>\mu(ca^{r-i_0}p_{i_0})=0
\end{equation}
and subtracting $\mu(ca^{r-i_0}p_{i_0})$ gives us
\begin{equation}\label{eqabkeypolnotwork3}
\mu\left(\sum_{i=0}^r\left(\frac{ca^{r-i}p_i}{ca^{r-i_0}p_{i_0}}\right)(aq)^i\right)>0.
\end{equation}
The equation above guarantees that
\begin{equation}\label{equantalgelemcoefinbt}
0=\left(\displaystyle\sum_{i=0}^r\left(\frac{ca^{r-i}p_i}{ca^{r-i_0}p_{i_0}}\right)(aq)^i\right)\mu=\sum_{i=0}^r\frac{(ca^{r-i}p_i)\mu}{(ca^{r-i_0}p_{i_0})\mu}((aq)\mu)^i.
\end{equation}
Since
\[
\frac{(ca^{r-i_0}p_{i_0})\mu}{(ca^{r-i_0}p_{i_0})\mu}=1\neq 0,
\]
equation (\ref{equantalgelemcoefinbt}) implies that $(aq)\mu$ is algebraic over $K\nu\left((ca^{r-i}p_i)\mu\mid i=0,\ldots,r\right)$ and by induction hypothesis also over $K\nu$. Now take any element $p/q\in K(x)$ with $\nu(p/q)\geq 0$. If $\nu(p)>\nu(q)$, then $\left(p/q\right)\nu$ is zero and hence algebraic over $K\nu$. If $\nu(p)=\nu(q)$, then there exists $a\in \overline K$ such that $\mu(ap)=\mu(aq)=0$. By the first part we have that $(aq)\mu$ and $(ap)\mu$ are algebraic over $K\nu$. Hence $(p/q)\nu=(ap/aq)\mu=(ap)\mu/(aq)\mu$ is algebraic over $K\nu$ and this concludes our proof.
\end{proof}

\section*{Appendix: The local uniformization problem}\label{Locunifprob}

\textit{Resolution of singularities} for an algebraic variety is an important topic in algebraic geometry. For an algebraic variety $X$ over a field $k$, a resolution of singularities is a modification (i.e., a proper birational morphism) $X'\lra X$ such that $X'$ has no singularities. \textit{Local uniformization} is the local version of resolution of singularities. Namely, for a valuation $\nu$ on $k(X)$ having a center $x$ in $X$, a local uniformization for $\nu$ is a modification $X'\lra X$ such that the center of $\nu$ in $X'$ is non-singular.

Local uniformization was introduced by Zariski in order to prove resolution of singularities. His approach consists of two steps: proving local uniformization for every valuation and use these local solutions to obtain a resolution of all singularities. Zariski succeeded in 1940 (see \cite{Zar}) in proving local uniformization for valuations having a center on any algebraic variety over a field of characteristic zero. He used this to prove resolution of singularities for algebraic varieties of dimension smaller or equal to three over a field of characteristic zero. Hironaka proved in 1964 (see \cite{Hir_1}) that every variety over a field of characteristic zero admits resolution of singularities without using valuations.

Abhyankar proved in 1956 (see \cite{Ab_4}), using Zariski's approach, resolution of singularities for algebraic varieties of dimension three and characteristic greater than five. In 2009, Cossart and Piltant concluded the proof (see \cite{Cos1} and \cite{Cos2}) of resolution of singularities for dimension three and any positive characteristic (they also used Zariski's approach). However, both resolution of singularities and local uniformization are open problems for algebraic varieties of dimension greater than three and positive characteristic.

Although the problem of local uniformization is still open, various programs for its resolution have gained strength in recent years. Two of them are those by Knaf and Kuhlmann using \textit{ramification theory} and by Spivakovsky using the \textit{theory of local blow-ups}. Knaf and Kuhlmann use pseudo-convergent sequences and Spivakovsky uses key polynomials in their respective programs. Since these objects appear in a similar way in the respective programs, it is important to understand their relation. This was partially done in Theorem 1.2 of \cite{SopivNova}. This paper provides a new insight in the search for a dictionary between these two programs. We believe that with such dictionary we will be able to obtain more precise results towards a full resolution of the local uniformization problem in positive characteristic.

Since local uniformization is a local problem, it can be reduced to local domains: a valuation centered on a local domain $(R,\MI)$ (i.e., for which $(R,\MI)$ is dominated by $(\VR_\nu,\MI_\nu)$) admits local uniformization if there exist $b_1,\ldots,b_n\in \VR_\nu$ such that $R^{(1)}:=R[b_1,\ldots,b_n]_{\MI_\nu\cap R[b_1,\ldots,b_n]}$ is regular (the map $R\lra R^{(1)}$ is called a \textbf{local blow-up of $R$ along $\nu$}). In Spivakovsky's program to solve the local uniformization problem, a minimal set of generators $(u_1,\ldots,u_d)$ of $\MI$ is fixed. Then, there exists an extension (which we call again $\nu$) of $\nu$ to the completion $k[[x_1,\ldots,x_d]]$ of $R$ with respect to $\nu$ (we assume that we are in the equicharacteristic case, i.e., that $k=R/\MI$ embeds in $R$). The goal is to prove that $\nu$ admits (a stronger version) of local uniformization on $k[[x_1,\ldots,x_d]]$ and use this to prove that $\nu$ admits local uniformization on $R$.

Now we can consider for each $i$, $0\leq i\leq d-1$ the extension of the induced valuation (which we call $\nu_i$) from $k((x_1,\ldots,x_i))$ to $k((x_1,\ldots,x_i))[x_{i+1}]$. In order to prove that $\nu$ admits local uniformization, it is important to understand the interaction between local blow-ups and the key polynomials that define the valuation $\nu_i$. This paper provides a way of interpreting key polynomials in terms of minimal pairs. We intend to use this to implement results from other areas (for instance, ramification theory) that have been used in similar settings (like those of Knaf and Kuhlmann) to Spivakovsky's program for solving the local uniformization problem in positive characteristic.

\noindent{\footnotesize JOSNEI NOVACOSKI\\
ICMC - USP\\
Av. Trabalhador S\~ao-Carlense, 400\\
13566-590 - S\~ao Carlos - SP\\
Email: {\tt jan328@mail.usask.ca} \\\\
\end{document}